\documentclass[12pt]{amsart}
\usepackage{amsfonts} 
\usepackage[all]{xy}
\usepackage{graphicx}
\usepackage{amssymb}
\usepackage{color}
\makeatletter
\newdimen\sectionruledimen
\def\makeparrule{
  \def\par{
   \endgraf\nobreak\vskip\lineskip\nointerlineskip
   \hbox to\hsize{\hskip\sectionbackskip\leaders\hrule height
   \sectionruledimen\hfil}
   }
 }

\def\section{%
 \@startsection {section}{1}{\sectionbackskip}{-10pt plus -1.2ex minus -.2ex}%
  {.5pt}{\normalsize\bf\makeparrule}%
 }%
\newdimen\sectionbackskip
\sectionbackskip\z@
\newdimen\sectionruledimen
\sectionruledimen\z@

\def\punteada{\leaders\hbox{$\m@th \mkern1.5mu - \mkern1.5mu$}\hfill}
\makeatother

\newtheorem{theorem}{Theorem}
\newenvironment{theor}{\smallskip\begin{trivlist}
   \item[\hspace{\labelsep}{\noindent\bf Theorem.}]\it
   }{\end{trivlist}\smallskip}

\newenvironment{lemma}{\smallskip\begin{trivlist}
   \item[\hspace{\labelsep}{\noindent\bf Lemma.}]\it
   }{\end{trivlist}\smallskip}

\newenvironment{propo}{\smallskip\begin{trivlist}
   \item[\hspace{\labelsep}{\noindent\bf Proposition.}]\it
   }{\end{trivlist}\smallskip}

\newcommand{\cuadro}{\hfill{$\qed$}}
\newcommand{\cuadrito}{\hbox{$\scriptstyle\sqcap \hskip-5.5pt \sqcup $}}

\def\lftcol#1{\vbox {\halign {##\hfil \cr #1\cr }}}
\long\def\direc#1#2{\hbox to \hsize{
     \lftcol{\hsize=7cm#1}\hfill\lftcol{\hsize=7cm#2}}}
\newcommand{\A}{\mathbb{A}}
\newcommand{\C}{\mathbb{C}}
\newcommand{\D}{\mathbb{D}}
\newcommand{\E}{\mathbb{E}}
\newcommand{\h}{\mathbb{H}}

\newcommand{\B}{\mathbb{B}}
\newcommand{\N}{\mathbb{N}}

\newcommand{\Oo}{\mathcal{O}}
\newcommand{\Q}{\mathbb{Q}}
\newcommand{\R}{\mathbb{R}}
\newcommand{\s}{\mathbb{S}}
\newcommand{\T}{\mathbb{T}}
\newcommand{\x}{\mathbb{X}}
\newcommand{\X}{\mathbb{X}}
\newcommand{\Z}{\mathbb{Z}}

\newcommand{\coh}{\hbox{\rm coh}}

\newcommand{\tr}{\hbox {\rm tr \,}}
\newcommand{\Der}{\hbox{\rm Der}}

\renewcommand{\dim}{\hbox{\rm dim}}
\newcommand{\udim}{\hbox{\bf dim \,}}
\newcommand{\Ext}{\hbox{\rm Ext}}
\newcommand{\End}{\hbox{\rm End}}

\newcommand{\Hom}{\hbox{\rm Hom}}

\renewcommand{\mod}{\hbox{\rm mod}}

\newcommand{\rad}{\hbox{\rm rad}}

\newcommand{\bulitito}{{\scriptscriptstyle \bullet}}
\def\subsetnoteq{\mathbin{\hbox{$\subseteq \joinrel \hskip-8pt \lower3pt
                 \hbox{$\scriptscriptstyle /$}\ $}}}

\def\raya{\raise1.5pt\hbox to 25pt{\vrule height1.5pt depth-1pt
           width25pt}}
\def\rayita{\raise2pt\hbox to 7.5pt{\vrule height1.5pt depth-1pt
           width7.5pt}}

\begin{document}

\title[Algebras of cyclotomic type]{Algebras whose Coxeter polynomials are products of cyclotomic polynomials.}
\author[Jos\'e A. de la Pe\~na]{Jos\'e A. de la Pe\~na}

\begin{abstract}  
Let $A$ be a finite dimensional algebra over an algebraically closed field $k$. Assume $A$ is basic connected with $n$ pairwise non-isomorphic simple modules. We consider the {\em Coxeter transformation} $\phi_A$ as the automorphism of the Grothendieck group $K_0(A)$ induced by the Auslander-Reiten translation $\tau$ in the derived category $\Der(\mod A)$ of the module category $\mod A$ of finite dimensional left $A$-modules. We say that $A$ is an algebra of {\em cyclotomic type} if the characteristic polynomial $\chi_A$ of $\phi_A$ is a product of cyclotomic polynomials. There are many examples of algebras of cyclotomic type in the representaton theory literature: hereditary algebras of Dynkin and extended Dynkin types, canonical algebras, some supercanonical and extended canonical algebras. Among other results, we show that: $(a)$ algebras satisfying the fractional Calabi-Yau property have periodic Coxeter transformation and are, therefore, of cyclotomic type, and $(b)$ algebras whose homological form $h_A$ is non-negative are of cyclotomic type. For an algebra $A$ of cyclotomic type we describe the shape of the Auslander-Reiten components of $\Der(\mod A)$.  
\end{abstract}
\maketitle

\section{Introduction.}
Assume throughout the paper that $k$ is an algebraically closed field. 
Let $A$ be a triangular finite dimensional $k$-algebra, that is, we assume that the quiver $Q_A$ of $A$ has no oriented cycles. Hence $A$ has finite global dimension. In this case, a theorem of Happel~\cite{Ha} asserts that the bounded derived category $\Der(A)=\Der^b(\mod A)$ of the category $\mod A$ of finite dimensional (left) $A$-modules has Serre duality of the form
$$
\Hom(X,Y[1])=\Hom(Y,\tau X)
$$
where $\tau$ is a self-equivalence of $\Der(A)$. In particular, $\Der(A)$ has almost-split triangles and the equivalence $\tau$ serves as the Auslander-Reiten translation of $\Der(A)$. In this setting, the {\em Grothendieck group} $K_0(A)$, formed with respect to short exact sequences, is naturally isomorphic to the Grothendieck group of the derived category, formed with respect to exact triangles. 
The Coxeter transformation $\phi_A$ is the automorphism of the Grothendieck group $K_0(A)$ induced by the Auslander-Reiten translation $\tau$ in the derived category $\Der(A)$. The characteristic polynomial $\chi_A(T)$ of $\phi_A$ is called the {\em Coxeter polynomial}  of $A$, and denoted simply $\chi_A$. It is a monic and self-reciprocal polynomial, i.e.
 $\chi_A(T)= T^n \chi_A(\frac{1}{T})=a_0 + a_1 T + a_2 T^2 + \ldots + a_{n-2} T^{n-2} + a_{n-1}T^{n-1} + a_n T^n \in \Z[T]$, 
where $a_0=1=a_n$. 
Sometimes, for the sake of clarity, we write $a_i(A)$ for these coefficients.  

Consider the roots $\lambda_1(A),\ldots,\lambda_n(A)$ of $\chi_A$, the so called  {\em spectrum} of $A$. There is a number of measures associated to the absolute values
$|\lambda|$ for $\lambda$ in the spectrum ${\rm Spec}(\phi_A)$ of $A$. Three important measures are the following:
the {\em spectral radius} of $A$ is defined as $\rho_A = {\rm max \, } \{|\lambda| : \lambda \in {\rm Spec}(\phi_A) \}$,
the {\em Mahler measure} of $\phi_A$ is $M(\phi_A)= \prod_{i=1}^n {\rm max} \{ 1, |\lambda_i|\}$ and the {\em energy function} 
of $\phi_A$ is $e(\phi_A)=\sum\limits_{i=1}^n |\lambda_i(A)|$. In Section 2 we show that  always $e(\phi_A) \ge n$.
Due to a celebrated Kronecker's theorem, we prove the following result.

\begin{theorem}
Let $A$ be a triangular algebra. The folllowing statements are equivalent:
\begin{enumerate} 
\item $|\lambda_i(A)| \le 1$, for all $1 \le i \le n$;
\item $|\lambda_i(A)| = 1$, for all $1 \le i \le n$;
\item $\rho_A =1$;
\item $M(\phi_A) = 1$;
\item $e(\phi_A)=n$
\item $\chi_A(T)$ factorizes as product of cyclotomic polynomials. 
\end{enumerate}
\end{theorem}

An algebra $A$ satisfying these equivalent conditions is said to be of {\em cyclotomic type}.

The following finite dimensional algebras are known (or shown) to be of cyclotomic type:
\begin{enumerate} 
\item  hereditary algebras of finite or tame representation type; 
\item tensor products of algebras of cyclotomic type;
\item quotient algebras $A_n(3)$ of the linear quiver
$${\footnotesize \xymatrix{1 \ar^x[r]&2 \ar^x[r]&\cdots \ar^x[r] &n}}$$
\noindent
with relations $x^3=0$, for $n$ even and many instances of $n$ odd;
\item all canonical algebras;
\item (some) extended canonical algebras;
\item (some) algebras whse module category is derived equivalent to categories of coherent sheaves;
\item  algebras whose homological bilinear form is non-negative.
\end{enumerate}
One of the purposes of this work is to show that several of the common properties of these classes of algebras are a consequence of the cyclotomic condition, see Sections 2, 5 and 6 for details.

Recall that a triangular algebra $A$ is said to be $\frac{p}{q}$-{\em Calabi-Yau} for integers $q \ge 1$ and $p \in \Z$ if $S^q=[p]$ in the derived category $\Der(A)$, where $S=\tau \circ[1]$. We show that algebras satisfying the fractional Calabi-Yau property have periodic Coxeter transformation and are, therefore, of cyclotomic type, see Sections 2 and 5.

Along this work we explore the structure of the Auslander-Reiten quiver $\Gamma_A$ of the derived category of an algebra $A$ of cyclotomic type. 
We introduce the {\em class quiver} $[\Gamma_A]$ of $A$ formed by the quotients $[{\mathcal C}]$ of components $\mathcal C$ of $\Gamma_A$, obtained by identifying $X, Y \in {\mathcal C}$ if $[X]=[Y]$ in $K_0(\Der(A))$. 
Among other results we prove the following theorem.

\begin{theorem}
Let $A$ be a triangular algebra with $n$ non-isomorphic simple modules and $\chi_A=\prod_{m \in M}\Phi_m^{e(m)}$ be an irreducible decomposition of its Coxeter polynomial. Let $\Gamma_A$ be the Auslander-Reiten quiver of the derived category $\Der(A)$. 
Then the following holds: 

(a) if $A$ satisfies the fractional Calabi-Yau property then $\phi_A$ is periodic;

(b) every component of the class quiver $[\Gamma_A]$ is a tube if and only if  $\phi_A$ is periodic. 

If $\phi_A$ is periodic then $A$ is of cyclotomic type and 

(c) the period is ${\rm l.c.m.}\{\phi(m) : m \in M\}$, where $\phi(-)$ is Euler's totient function;

(d) every component of $\Gamma$ is either a tube $\Z \A_\infty / (p)$ of finite period $p$ where $p=\phi(m)$ for some $m \in M$ or of the form $\Z \Delta$ for $\Delta$ a Dynkin or extended Dynkin diagram or of one of the shapes $\A_\infty$, $\A_\infty^\infty$ or $\D_\infty$. Moreover,

(e) if $p_1,\ldots,p_s$ are the periods $> 1$ of non-homogeneous tubes in $[\Gamma]$ then $\sum\limits_{i=1}^s p_i \le n$.
\end{theorem}

Along the paper, we freely use specialized terminology: tame algebras, wild algebras, linear growth,
polynomial growth, pg-critical, and others. We refer readers to the list of references where all these and related concepts are explained, see \cite{ASS,LePe3, Pe, Ri}. Results contained in this work were presented in conferences at Banff, Canada (2010), Bielefeld, Germany and Guanajuato, M\'exico (2012).

\section{The spectrum of algebras of cyclotomic type} 

\subsection{}
If the spectrum of $A$ lies in the unit disk, then all roots of $\chi_A(T)$ lie on the unit circle, hence $A$ has spectral radius $\rho_A=1$. More precisely, we recall the following famous theorem of Kronecker.

\begin{theor}  \cite {Kr}
Let $f$ be a monic integral polynomial whose spectrum is contained in the unit disk. Then all roots of $f$ are roots of unity or $0$. Equivalently, $f$ factors into cyclotomic polynomials and $T^m$, for some integer $m \ge 0$.
\end{theor}

\subsection{} 
We recall some facts on {\em cyclotomic polynomials}.

The $n$-cyclotomic polynomial $\Phi_n$ is inductively defined by
the formula
$${\textstyle T^n-1=\prod\limits_{d\mid n}\Phi_d(T).}$$
The roots of $\Phi_n$ are primitive $\phi(n)$-roots of unity, where $\phi(-)$ denotes Euler's totient function.

The {\em M\"obius function\/} is defined as follows:
$$\mu (n)=\begin{cases}
0 &\hbox{if $n$ is divisible by a square}\\
(-1)^r &\hbox{if $n=p_1,\ldots p_r$ is a factorization into distinct
primes.}\end{cases}$$ 
\noindent 
A more explicit expression for the cyclotomic polynomials is given by:

\begin{lemma}
For each $n\ge 2$, we have
$${\textstyle \Phi_n=\prod\limits_{1\le d<n\atop d\mid n}v_{n/d}^{\mu (d)}}$$
\cuadro
\end{lemma}

In the lemma, we set  $v_n=1+T+T^2+\ldots+T^{n-1}$. Note that $v_n$ has degree $n-1$.

\subsection{}
A {\em path algebra} $k\Delta$ is said to be of {\em Dynkin type}
if the underlying graph $|\Delta|$ of $\Delta$ is one of the {\em ADE-series}, that is, of type, $\A_n, \D_n$, for some $n \ge 1$ or $\E_p$, for $p=6,7,8$.
The corresponding factorization of the Coxeter polynomial $\chi_{k\Delta}$ is as follows.
\small
\begin{center}
\renewcommand\arraystretch{1.4}
\begin{tabular}{|l|c|c|c|c|} \hline
  Dynkin  & star &$v$-factorization & cyclotomic & Coxeter \\
  type & symbol &  & factorization & number \\ \hline
  $\A_{n}$& $[n]$ & $v_{n+1}$ & $\displaystyle\prod_{d|n,d>1}\Phi_d$ & $n+1$\\ \hline
  $\D_{n}$&$[2,2,n-2]$& $\displaystyle\frac{v_2\,(v_2v_{n-2})}{(v_2v_{n-2})v_{n-1}}v_{2(n-1)}$ & $\displaystyle\Phi_2\,\prod_{\frac{d|2(n-1)}{d\neq 1,d\neq n-1}}\Phi_{d}$ & $2(n-1)$ \\\hline
  $\E_{6}$ &$[2,3,3]$& $\displaystyle\frac{v_2 v_3 (v_3)}{(v_3)v_4v_6}v_{12}$ & $\Phi_3\Phi_{12}$& 12 \\ \hline
  $\E_{7}$&$[2,3,4]$& $\displaystyle\frac{v_2v_3(v_4)}{(v_4)v_6 v_9}v_{18}$& $\Phi_2\Phi_{18}$& 18 \\\hline
  $\E_{8}$& $[2,3,5]$ &$\displaystyle\frac{v_2v_3v_5}{v_6v_{10}v_{15}}v_{30}$&$\Phi_{30}$&30\\\hline
\end{tabular}
\end{center}
\normalsize
In the column `$v$-factorization', we have added some extra terms in the nominator and denominator which obviously cancel.

Inspection of the table shows the following result:

\begin{propo} \label{prop:dynkin}
Let $k$ be an algebraically closed field and $A$ be a connected, hereditary $k$-algebra which is representation-finite. Then the Coxeter polynomial $\chi_A(T)$ determines $A$ up to derived equivalence.
\end{propo}

\subsection{}
Let $A$ be the path algebra of a hereditary star $[p_1,\ldots,p_t]$
with respect to the standard orientation, for instance
 $$
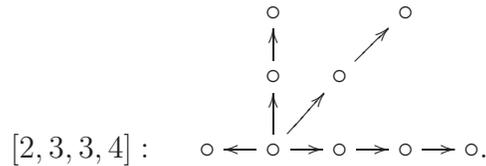
\def\c{\circ}
\xymatrix@C12pt@R12pt{
        &&\c                          &        &\c  &\\
        &&\c\ar[u]                    &\c\ar[ru]&     &\\
[2,3,3,4]:&\c      &\c\ar[l]\ar[u]\ar[r]\ar[ru]&\c\ar[r]&\c\ar[r]&\c.\\
}
$$
Since the Coxeter polynomial $\chi_{A}$ does not depend on the orientation of $A$ we will denote it by $\chi_{[p_1,\ldots,p_t]}$. It follows from \cite[Prop.~9.1]{LePe3} that
\begin{equation}\label{eq:star_formula}
\chi_{[p_1,\ldots,p_t]}=\prod_{i=1}^{t}v_{p_i}\left((T+1)-T\,\sum_{i^=1}^t \frac{v_{p_i-1}}{v_{p_i}} \right).
\end{equation}
In particular, we have an explicit formula for the sum of coefficients of $\chi_{[p_1,\ldots,p_t]}$ as follows:
$$
\sum\limits_{i=0}^n a_i=  \chi_{[p_1,\ldots,p_t]}(1)=\prod_{i=1}^t p_i\left( 2-\sum_{i=1}^t (1-\frac{1}{p_i})\right).
$$
This special value of $\chi=\chi_{[p_1,\ldots,p_t]}$ has a specific mathematical meaning:
up to the factor $\prod_{i=1}^t p_i$ this is just the orbifold-Euler characteristic of a weighted projective line $\X$ of weight type $(p_1,\ldots,p_t)$. Moreover,
\begin{enumerate}
\item $\chi(1)>0$ if and only if the star $[p_1,\ldots,p_t]$ is of Dynkin type, correspondingly the algebra $A$ is representation-finite.
\item $\chi(1)=0$ if and only if the star $[p_1,\ldots,p_t]$ is of extended Dynkin type, correspondingly the algebra $A$ is of tame (domestic) type.
\item $\chi(1)<0$ if and only if $[p_1,\ldots,p_t]$ is not Dynkin or extended Dynkin, correspondingly the algebra $A$ is of wild representation type.
\end{enumerate}
To complete the picture, we consider the cyclotomic factorization of the extended Dynkin quivers without oriented cycles. Observe that, in one case, the Coxeter polynomial depends on the orientation: If $p$ (resp.\ $q$) denotes the number of arrows in clockwise (resp.\ anticlockwise) orientation ($p,q\geq1, p+q=n+1$), that is, the quiver has type ${\tilde \A}_{(p,q)}$, the Coxeter polynomial is $\chi_{(p,q)}=(T-1)^2\,v_{p}v_{q}.$.

The next table displays the $v$-factorization of extended Dynkin quivers.
\bigskip
\begin{center}
\renewcommand\arraystretch{1.3}
\begin{tabular}{|c|c|c|c|}\hline
extended Dynkin type & star symbol & weight symbol & Coxeter polynomial \\ \hline
${\tilde \A}_{p,q}$           &     ---     & $(p,q)$         & $(T-1)^2v_p\,v_q$ \\ \hline
${\tilde \D}_{n}$, $n\geq4$   & [2,2,n-2]   & $(2,2,n-2)$   & $(T-1)^2v_2^2v_{n-2}$\\\hline
${\tilde \E}_{6}$& $[3,3,3]$ & $(2,3,3)$    & $(T-1)^2 v_2 v_3^2$ \\\hline
${\tilde \E}_{7}$& $[2,4,4]$ & $(2,3,4)$    & $(T-1)^2 v_2v_3v_4$ \\\hline
${\tilde \E}_{8}$& $[2,3,6]$ & $(2,3,5)$    & $(T-1)^2 v_2v_3v_5$\\\hline
\end{tabular}
\end{center}

The above remark on representation-finite hereditary algebras extends to the tame hereditary case. That is, the Coxeter polynomial of a connected, tame hereditary $k$-algebra $A$ ($k$ algebraically closed) determines the algebra $A$ up to derived equivalence. This is no longer true for wild hereditary algebras, not even for trees.

\subsection{}
Recall that $\chi_A$ is a self-reciprocal polynomial, that is, $T^n \chi_A({1 \over T})=\chi_A(T)$.
Moreover, it was shown in \cite{LePe3} that $\chi_{A}(-1)$ is the square of an integer.

These conditions prevent many polynomials of being Coxeter polynomials of an algebra. For instance, the cyclotomic polynomials $\Phi_4$, $\Phi_6$, $\Phi_8$, $\Phi_{10}$. Moreover, the polynomial $f(T)=T^3+1$, which is the Coxeter polynomial of the non simply-laced Dynkin diagram $\B_{3}$,  does not appear as the Coxeter polynomial of a triangular algebra over an algebraically closed field, despite  the fact that $f(-1)=0$ is a square, see  \cite{LePe3}.

\subsection{} 
The {\em homological} quadratic form $h_A$ is the symmetrization of the inverse Cartan matrix $C_A$ associated to the algebra $A$, that is $h_A(x)=x^T(C_A^{-1}+C_A^{-T})x$ for $x \in K_0(A)$. Write $\langle u,v \rangle = u^T C_A^{-1}v$ which yields a (non-symmetric) non-degenerated bilinear form, that is, if $\langle u,v \rangle=0$ for a fixed $u$ and all $v$ then $u=0$. Mooreover $\langle x,\phi_A(y) \rangle_A= - \langle y,x \rangle_A$. Therefore $\phi_A= -C_A^{-t}C_A$.

For any two $A$-modules $X,Y$, with classes $\udim X, \udim Y$ in the Grothendieck group $K_0(A)$ we get $\langle \udim X, \udim Y \rangle_A = \sum\limits_{j=0}^{\infty} (-1)^j {\rm dim}_k \Ext_A^j(X,Y)$. 

Consider the  quadratic form $h_A(x) = \langle x,x \rangle$. It is shown in \cite{Pe2} that
properties of $h_A$ characterize the representation behavior of $A$. Indeed, $h_A$ is positive definite
(resp. non-negative of corank $1$) if and only if $A$ is derived equivalent to the path algebra $k \Delta$, where $\Delta$ is a quiver with underlying graph a Dynkin diagram (resp. extended Dynkin diagram). There are algebras $A$ with $h_A$ non-negative of arbitrary corank. 

\begin{lemma}
If $u$ (resp. $v$) is an eigenvector of $\phi_A$ with eigenvalue $\lambda$ (resp. $\mu$)
such that $\lambda \ne \mu^{-1}$ then $\langle u,v \rangle=0$. Moreover, if $\lambda \ne -1$ then $h_A(u)=0$.
\end{lemma}
\begin{proof}
Assume $C_A^{-t}u=-\lambda C_A^{-1}u$ and $C_A^{-t}v=-\mu C_A^{-1}v$. Then
$$\langle v,u \rangle =v^t C_A^{-t}u=-\lambda v^t C_A^{-1}u=\lambda \mu v^t C_A^{-t}u$$
\noindent
and the first claim follows. Moreover if $\lambda^2 \ne 1$ then $h_A(u)=0$. The case $\lambda=1$
yields $\langle u,u \rangle + \langle u,u \rangle =0$ or equivalently, $h_A(u)=0$. 
\end{proof}

\subsection{}
For an algebra $A$ and a finite-dimensional right $A$-module $M$ we call 
$$A[M]=\left[\begin{array}{cc}
A & 0\\ M &k
\end{array}\right]$$
\noindent
the {\em one-point extension} of $A$ by $M$, see \cite{Ri}. This construction provides an order to deal inductively with {\em triangular algebras}.

For $A=B[M]$, Happel's long exact sequence \cite{Ha} relates the Hochschild cohomology groups $H^i(A)$ and $H^i(B)$ in the following way:
$$0\!\to\! H^0(A)\!\to\! H^0(B)\!\to\! \End_B(M)/k\!\to\! H^1(A)\!\to\! H^1(B)\!\to\! \Ext^1_B(M,M)\!\to\! H^2(A)\!\to\! \cdots$$

Therefore, if $M$ is exceptional we get $H^i(A)=H^i(B)$, for $i\ge 0$, and moreover, the cohomology rings $H^*(A)$ and $H^*(B)$ are isomorphic. More general, if $A$ is accessible from $B$,
that is, there exists a chain of algebras $A_1=B, A_2,\ldots, A_s=A$ such that $A_{i+1}$ is a one-point extension or coextension of $A_i$ by an exceptional $A_i$-module $M_i$, then there is a ring isomorphism $H^*(A)\cong H^*(B)$. In particular, if $A$ is accessible (from $k$) then $H^i(A)=0$ for $i>0$ and $H^0(A) =k$.

An important expression for the linear term of $\chi_A(T)$ was shown in \cite{Ha} as:
$a_1= \sum\limits_{i=0}^n (-1)^i \dim_k H^i(A).$ 

\subsection{Accessible algebras} 
Following \cite{LePe5}, for an algebra $B$ we say that $A$ is {\em accessible from\/} $B$ if there is a sequence $B=B_1,B_2,\ldots,B_s=A$ of algebras such that each $B_{i+1}$ is a one-point extension (resp.\ coextension) of $B_i$ for some exceptional $B_i$-module $M_i$.

As a special case, a $k$-algebra $A$ is called {\em accessible} if $A$ is accessible from $k$. By construction, each accessible algebra $A$ is connected, moreover the indecomposable projective  $A$-modules can be arranged to form  an exceptional sequence $(P_1,\ldots,P_n)$, that is, $\Ext_A^i(P_j,P_s)=0$ for every $j > s$ and $i \ge 0$. In particular, the quiver of an accessible algebra $A$ has no loops or oriented cycles (we say that $A$ is {\em triangular}) and therefore $A$ has finite global dimension.

If not stated otherwise we shall assume that an algebra $A$ is defined over $k$ and {\em connected}, that is, its quiver is connected. Many well-known examples of algebras are accessible: hereditary algebras of tree type, more generally tree algebras, canonical algebras with three weights, poset algebras without crowns, representation-finite algebras with vanishing first Hochschild cohomology of global dimension $\leq2$. Moreover, extended canonical algebras with three weights and certain supercanonical algebras also are accessible. The construction of towers of algebras $A=A_n,A_{n-1},\ldots,A_1=k$ where each algebra $A_{i+1}$ is a one-point extension or coextension of $A_i$ by an exceptional module $M_i$, for $i=1, \ldots,n-1$ allows to perform informative inductive procedures. For instance, every accessible algebra $A$ has vanishing Hochschild cohomology $H^i(A)=0$, for $i>0$ and $H^0(A)=k$. Also, for any tree algebra $A$ such an inductive procedure yields that the Coxeter polynomial takes values $0$ or $1$ when evaluated at $-1$.

\subsection{The Calabi-Yau property}  
Let $\T$ be a triangulated category with Serre functor $S$, that is, we have functorial isomorphisms 
$$D\Hom(X,Y) \cong \Hom(Y,SX)$$
\noindent
for $X,Y \in \T$.

Denote by $[1]$ the {\em suspension functor} on $\T$. 
If $\tau$ is the Auslander-Reiten translation in $\T$, then $S=\tau \circ[1]$.
The category $\T$ is said to be $\frac{p}{q}$-{\em Calabi-Yau} for integers $q \ge 1$ and $p \in \Z$ if $S^q=[p]$. If additionally $q$ is chosen minimal, then we say that $\T$ has {\em Calabi-Yau dimension} $\frac{p}{q}$. Notation: {\rm CY-dim \,}$\T = \frac{p}{q}$.

{\em Example:} Let $E$ be a smooth elliptic curve. Then the category $\coh E$ of coherent sheaves on $E$ has Serre duality in the form
$D\Ext^1(X,Y) = \Hom(Y,X)$, yielding $\tau= 1$ and $S=[1]$. Thus $\coh E$ is Calabi-Yau of dimension $\frac{1}{1}$. In this case we say that
$\coh E$ is CY.

The following is well-known.

\begin{lemma}
Assume that the triangulated category $\T$ is $\frac{p}{q}$-Calabi-Yau.
Let $\sigma$ be the $k$-linear automorphism of the Grothendieck group $K_0(\T)$ induced by the Serre functor $S$. 
Then the following holds:

(a) $\sigma^{2q}=1$;

(b) $\phi^{2q}=1$, where $\phi= - \sigma$ is the Coxeter transformation.
\end{lemma}
\begin{proof}
Use that the suspension $[1]$ induces on $K_0(\T)$ the map $-{\rm id}$.
\end{proof}

\subsection{} 
We say that an algebra $A$ is $\frac{p}{q}$-{\em Calabi-Yau} (or $A$ satisfies the CY-property) if $\Der(\mod \, A)$ has the corresponding property. 
The remarks above show that if $A$ satisfies the CY-property then $\phi_A$ is periodic, that is, 
$\phi_A^r={\rm id}$, for some $r \ge 1$. More generally, we have the following result.

\begin{theor} 
The following conditions are equivalent for $A$:

(1) $\phi_A$ is periodic;

(2) $A$ is of cyclotomic type and $\phi_A$ is diagonalizable;

(3) assume that $\chi_A=\prod_{m \in M}\Phi_m^{e(m)}$ for some subset $M$ of $\N$ and integers $e(m) \ge 1$, 
then $\phi_A$ has period ${\rm l.c.m.}\{ \phi(m) : m \in M\}$.
\end{theor}
\begin{proof}
(1) implies (2): Consider the Jordan block decomposition $J_{n_1}(\mu_1)\oplus \ldots\oplus J_{n_s}(\mu_s)$ of $\phi_A$. The $r$-th power of a block $J_m(\mu)^r$ is the identity if and only if $m=1$ and $\mu^r=1$. Hence if $\phi_A$ is periodic then $n_1=\ldots=n_s=1$ and $A$ is of cyclotomic type. 

(2) implies (1): Conversely, assume $P$ is an invertible matrix such that $P\phi_A P^{-1}={\rm diag}(\lambda_1,\ldots,\lambda_n)$ and $\lambda_s^r=1$, for all $s=1,\ldots,n$. Then $P\phi_A^rP^{-1}=(P\phi_AP^{-1})^r={\rm id}_n$ and $\phi_A^r= P^{-1}P={\rm id}_n$.

(3) implies (1) is clear. (1) implies (3): consider $\phi_A$ as an operator of the $n$-dimensional complex vector space $V$. For $\lambda$ an eigenvalue of $\phi_A$ there is a set of eigenvectors 
$B(\lambda)$ of $\lambda$ such that $\bigcup_{\lambda} B(\lambda)$ is a basis of $V$. Moreover, for
$\lambda$ there is some $m(\lambda) \in M$ such that $\Phi_{m(\lambda)}(\lambda)=0$. In particular, $\lambda^{\phi(m(\lambda))}=1$.

Set $r={\rm l.c.m.} \{ \phi(m): m \in M \} $ then for $u \in B(\lambda)$ we have $r= \phi(m(\lambda))r'$ and $\phi_A^r(u)=(\lambda^{\phi(m(\lambda))})^{r'}u=u$, that is, $\phi_A^r={\rm id}_V$.
\end{proof}

\section{Measures for $\chi_A$} 

\subsection{The spectral radius of a Coxeter transformation} 
Let $A$ be a basic finite dimensional $k$-algebra with $k$ an algebraically
closed field. We assume $A=kQ/I$ for a quiver $Q$ without
oriented cycles and $I$ an ideal of the path algebra.
The following facts about the Coxeter transformation $\phi_{A}$ of $A$ are fundamental:

(i) Let $S_1,\ldots,S_n$ be a complete system of pairwise
non-isomorphic simple $A$-modules, $P_1,\ldots,P_n$ the corresponding
projective covers and $I_1,\ldots,I_n$ the injective envelopes. Then
$\phi_{A}$ is the automorphism of $K_0(A)$ defined by $\phi_{A}[P_i]=-[I_i]$, where $[X]$
denotes the class of a module $X$ in $K_0(A)$.

(ii) For a hereditary algebra $A=kQ$, the spectral radius
$\rho_A$ determines the representation type of $A$ in the following manner:

$\bullet$ $A$ is representation-finite if $1=\rho_A$
is not a root of the Coxeter polynomial.

$\bullet$ $A$ is tame if $1=\rho_A \in {\rm Spec}(\phi_{A})$.

$\bullet$ $A$ is wild if $1< \rho_{A}$. Moreover, if $A$ is connected,
Ringel \cite{Ri2} shows that the spectral radius is a simple root of
$\chi_{A}$. Then Perron-Frobenius theory yields a vector $y^+\in
K_0(A)\otimes_{\Z}\R$ with positive coordinates such
that $\phi_{A}y^+ =\rho_{A}y^+$.  

\subsection{The energy of a matrix}
Let $A$ be an algebra and $\lambda_1, \ldots, \lambda_n$ be the eigenvalues  of the Coxeter transformation $\phi_A =(f_{st})_{1 \le s,t \le n}$.
We may assume that $|\lambda_1| \le |\lambda_2| \le \ldots \le |\lambda_n|=\rho_A$.

We recall that the {\em singular eigenvalues} of $\phi_A$ are the non-negative square roots of the eigenvalues of the symmetric matrix $\phi_A^T \phi_A$, 
which is positive semidefinite. We may assume that the singular eigenvalues are 
$$0 \le \sigma_1 \le \sigma_2 \le \ldots \le \sigma_n$$  
\noindent
The following holds,

(a) $\sigma_i \ge |\lambda_i|$, for all $i=1,\ldots,n$

(b) $\sum\limits_{i=1}^n |\lambda_i|^2 \le \tr(\phi_A^t \phi_A)= \sum\limits_{i=1}^n \sigma_i^2 =\sum\limits_{s,t=1}^n |f_{st}|^2$, 
equality holds if and only if $\phi_A$ is normal. We write 
$$||\phi_A||^2=\sum\limits_{i=1}^n \sigma_i^2 =\sum\limits_{s,t=1}^n |f_{st}|^2$$

\begin{theor}
The following inequatilities hold:
$$n \le e(\phi_A) \le {\sqrt n}\, ||\phi_A|| \le n \rho_A$$
\noindent
Moreover, $A$ is of cyclotomic type if and only if any of the equalities hold.
\end{theor}
\begin{proof}
We start considering the first lower bound.
Let $\lambda$ be an eigenvalue of $\phi_A$. Since $\phi_A$ is invertible, $\lambda \ne 0$. Since the characteristic polynomial $\chi_A$ is
real (resp. self-reciprocal) then ${\overline \lambda}$ (resp. $\lambda^{-1}$) is another root of $\chi_A$. Consider the set $R(s)= \{ \lambda_s, {\overline \lambda_s}, \lambda_s^{-1}, {\overline \lambda_s^{-1}}\}$ the roots associated with $\lambda_s$. For $\lambda_s=a_s +i b_s \ne 0$, for real numbers $a_s,b_s$, 
the cardinality $r(s)$ of $R(s)$ may be $4$ (in case $b_s \ne 0$), or $2$ (in case $b_s=0$ and $a_s \ne 0,1$) or $1$ (in case $\lambda_s=1$). 

We shall prove that 
$\mu_s:= \sum\limits_{\mu \in R(s)} |\mu| \ge r(s)$ with equality if and only if $\lambda_s$ has modulus one. The wanted bound follows.

Indeed, write $\lambda_s=c(a+ib) \ne 0$ in {\em polar form}, that is $a,b,c$ are real numbers with $c>0$ and $a^2+b^2=1$. In case $r(s)=2$ then
$\mu_s=c + c^{-1}$, which is an increasing function at the interval $[1,\infty)$. Therefore with a minimal value $2$ at $1$. In case $r(s)=4$ then
$$\mu_s = (|c(a+ib)|+|c^{-1}(a+ib)|)+(|c(a-ib)|+|c^{-1}(a-ib)|)\ge 4$$
\noindent
with equality if and only if $c=1$.

Next, an application of Schwartz inequality yields,
$$e(A)^2 =  (\sum\limits_{i=1}^n |\lambda_i|)^2 \le  n \sum\limits_{i=1}^n |\lambda_i|^2  = n ||\phi_A||^2$$
\noindent
Moreover,
$$||\phi_A||^2=\sum\limits_{i=1}^n |\lambda_i|^2 \le n \rho_A^2$$
\noindent
Finally, if $A$ is of cyclotomic type then all $\lambda_i$ have modulus one. Conversely, if
$e(A) = n$, then $\mu_s= \sum\limits_{\mu \in R(s)} |\mu| = r(s)$ and 
$|\lambda_s|=1$, for all $s=1,\ldots, n$.  If it is the case that $e(A)={\sqrt n} ||\phi_A||$ then
$$(\sum\limits_{i=1}^n |\lambda_i|)^2 =  n \sum\limits_{i=1}^n |\lambda_i|^2$$
\noindent
which implies that all $|\lambda_i|$ have a common value. Since the set of eigenvalues of $\phi_A$ is closed under multiplicative inverses,
then this common value may only be $1$. In any case, $A$ is of cyclotomic type.
\end{proof}

{\bf Remark:} Alternatively, the first inequality in the Theorem above may be seen as a consequence of the arithmetic mean-geometric mean inequality. Indeed, for the set of non-negative numbers 
$\{ |\lambda_1|,\ldots,|\lambda_n|\}$ we get
$$e(A)=n \frac{\sum\limits_{i=1}^n |\lambda_i|}{n}  \ge n \left( \prod_{i=1}^n |\lambda_i| \right)^{\frac{1}{n}}=
n (|{\rm det \,}(\phi_A)|)^{\frac{1}{n}}=n.$$

\subsection{Mahler measure}
Given a Laurent polynomial $P(z)$ with integer coefficients, its Mahler measure $M(P)$ is defined as the
geometric mean of the function $|P|$ over the real circle, that is,
$$M(P) = exp(\int_0^1 log (|P(e^{2 \pi i t})|) dt$$
\noindent
For a polynomial in one variable, Mahler obtained the more  elementary expression for the measure given in the Introduction, see \cite{Ma}.

Kronecker's result implies that $M(P(T)) = 1$ precisely when $P(T)$ is a product of cyclotomic polynomials and a monomial of the form
$T^m$. Rather little is known, however, about values of the Mahler measure near 1. In 1933, D. H. Lehmer  found that
the polynomial
$$T^{10} + T^9 - T^7 - T^6 - T^5 - T^4 - T^3 + T + 1$$
\noindent
has Mahler measure $\mu_0 = 1.176280 . . .$, and he asked if there exist any smaller values  exceeding 1. This
question of determining whether 1 is a limit point for the Mahler measure is known as {\em Lehmer's problem}.

In a forthcoming publication we show that for an accessible algebra $A$ not of cyclotomic type there is a convex subcategory 
$B$ of $A$ satisfying the following properties:

(a) $B$ is minimal not of cyclotomic type, that is, if $C$ is any proper convex subcategory of $B$,  then $C$ is of cyclotomic type;

(b) the Mahler measure of $B$ is $M(\chi_B) \ge \mu_0$.

\subsection{Proof of Theorem 1} The equivalence of (1), (2), (4) and (6) follows from Kronecker's theorem. The remaining equivalences follow from
(3.2). \cuadro

\section{On the decomposition of the Coxeter polynomial of an algebra of cyclotomic type}

\subsection{}
Let $A$ be an algebra with $n$ vertices whose Coxeter polynomial decomposes as $\chi_A=\prod_{m \in M}\Phi_m(T)^{e(m)}$ for some subset $M$ of $\N$ and integers $e(m) \ge 1$. Several conditions on the coefficients of $\chi_A= \sum\limits_{i=0}^n a_i T^i$ and of the polynomials $\Phi_m$ follow, for instance:

(1) $\sum\limits_{m \in M}e(m) \phi(m)=n$,  which follows from the degree calculation, where $\phi(-)$ is {\em Euler's totient function};
 
(2) if $1 \in M$ then $e(1)$ is an even number, which follows from the fact that $\Phi_m$ for $m \ne 1$, as well as $\chi_A(T)$,
are self-reciprocal polynomials, but $\Phi_1(T)=T-1$ is not so, while $\Phi_1(T)^2=T^2-2t+1$ is self-reciprocal;

(3) $\sum\limits_{m \in M}e(m) \mu(m) = - a_1$, which follows from the fact that the linear coefficient of $\Phi_m(T)$ is 
$a_1(\Phi_m(T))=-\mu(m)$. Use that $a_0(\Phi_m(T))=1$ for $m \ne 1$ and $a_0(\Phi_1(T)^{e(1)})=1$.

\begin{theor}
Let $A$ be an algebra whose Coxeter polynomial factorizes as $\chi_A(T)=\prod_{m \in M}\Phi_m(T)^{e(m)}$ for some subset $M$ of $\N$ and integers $e(m) \ge 1$. Set $\epsilon(m)=0$ if $m \in M$ and $\epsilon(m)=1$ if $m \notin M$. For any prime number $p$, let $M(p)$ (resp. $M^{\prime}(p)$) be the set of those $m \in M$ of the form $m=p^s$ (resp. $m=2 p^s$) for some $1 \le s$.  
Let $f(p)= \sum\limits_{p^s=m \in M(p)} e(m)$ and $f'(p)= \sum\limits_{2 p^s=m \in M'(p)} e(m)$.
Then

{\rm (a)} $\chi_A(1)=\epsilon(1)\, \prod_{m \in M(p)} p^{e(m)}= \epsilon(1) \, \prod_{p \,\,{\rm prime}} p^{f(p)}$

{\rm (b)} $\chi_A(-1)=\epsilon(2)\, 2^{e(1)} \prod_{m \in M^\prime(p)} p^{e(m)} =
\epsilon(2)\, 2^{e(1)} \prod_{p \,\, {\rm prime}} p^{f'(p)}$. 
Moreover, if this number is not zero then, for each prime $2 \le p$, the number
$f'(p)$ is even.
\end{theor}
\begin{proof}
(a) According to \cite{Pra}, the evaluation $\Phi_m(1)=1$ if $m \ne 1$ and $m \ne p^s$, for all prime $p$. Otherwise, $\Phi_1(1)=0$
or $\Phi_{p^s}(1)=p$. Then the evaluation 
$$\chi_A(1)= \prod_{m \in M}\Phi_m(1)^{e(m)}= \epsilon(1) \, \prod_{m \in M(p)} p^{e(m)}$$

(b) According to \cite{Pra}, the evaluation $\Phi_m(-1)$ yields 
$\Phi_m(-1)=-2$ if $m=1$, $\Phi_m(-1)= 0$ if $m=2$, $\Phi_m(-1)=p$ if $m=2 p^s$ with $p$ a prime number and $s \ge 1$ and $\Phi_m(-1)=1$ otherwise. Therefore
$$\chi_A(-1)= \prod_{m \in M}\Phi_m(-1)^{e(m)}= \epsilon(2) \, 2^{e(1)} \prod_{m \in M^{\prime}(p)} p^{e(m)}$$
\noindent
Assume $\chi_A(-1)=r^2 >0$ for some integer $r$, then for any prime $2 <p$
$$\prod_{m \in M^{\prime}(p)} p^{e(m)}= p^{\sum\limits_{m \in M^\prime(p)} e(m)}$$ 
\noindent
is an even power of $p$. For $p=2$, use additionally that $e(1)$ is an even number. The claim follows.
\end{proof}

\section{More examples}

\subsection{Tensor products of polynomials and algebras}

Given two monic integral polynomials $f(T)$ with roots $a_1,\ldots,a_n$ and $g(T)$ with roots $b_1,\ldots,b_m$, 
the {\em tensor product} of $f(T)$ and $g(T)$ is
$$f(T) \otimes g(T)= \sum\limits_{i=1}^n \sum\limits_{j=1}^m (T- a_i b_j)$$
\noindent
By \cite{Glasby}, the polynomial $f(T) \otimes g(T)$ is integral of degree ${\rm degree}(f(T)){\rm degree}(g(T))$.
An important problem is to find the irreducible decomposition of $f(T) \otimes g(T)$.

We are interested in calculating $\Phi_6(T) \otimes \Phi_m (T)$ for $m \in \Z$. It is easy to show that

$\Phi_m(T) \otimes \Phi_r(T)=\Phi_{mr}(T)$, whenever $m,r$ are relative prime.

We observe the following facts:

$\bullet$ $\Phi_6(T) \otimes \Phi_2(T)=\Phi_3(T)$

$\bullet$ $\Phi_6(T) \otimes \Phi_3(T)=\Phi_1(T)^2 \Phi_2(T)^2$ 

$\bullet$ $\Phi_6(T) \otimes \Phi_4(T)=\Phi_{12}(T)$

$\bullet$ $\Phi_6(T) \otimes \Phi_6(T)=\Phi_2(T)^2 \Phi_3(T)$

$\bullet$ $\Phi_6(T) \otimes \Phi_{12}(T)=\Phi_4(T)^2 \Phi_{12}(T)$

As it is easy to verify. For instance, we check the two last claims. Let $a$ and $b$ be the roots of $\Phi_6(T)=T^2-T+1$. The roots of $\Phi_6(T) \otimes \Phi_6(T)$ are 
$a^2=a-1=-b$, $b^2=-a$ and $ab=1=ba$.  
Let $c_i={\rm exp}({2 \pi i \over 12})$, for $i=1,5,7,11$, be the roots of $\Phi_{12}(T)=T^4-T^2+1$.
The roots of $\Phi_6(T) \otimes \Phi_{12}(T)$ are $ac_1, ac_7$ and $bc_5,bc_{11}$ which are roots of
$\Phi_4(T)$ and $ac_5=c_7,ac_{11}=c_1,bc_1=c_{11},bc_7=c_5$.

According to \cite{Glasby}, for $m=2^r 3^s q$, where $2$ and $3$ do not divide $q$, we have:

$\bullet$ $\Phi_6(T) \otimes \Phi_m(T)=\Phi_m(T)^2$, in case $r >1$ and $s>1$, 
\noindent
other products have more complex expressions.

\subsection{} As observed by Ladkani, see \cite{La} and also \cite{Ha3}, at the level of algebras we have:

\begin{propo}
Let $A$ and $B$ be algebras of cyclotomic type. Then $A \otimes B$ is an algebra of cyclotomic type with
$$\chi_{A\otimes B} =\chi_A \otimes \chi_B$$
\end{propo}

As a particular example, the  algebra $R_{2n}$  with $2n$ vertices and whose quiver is given as
$$\xymatrix{1\ar[r]\ar[d]&2\ar[r]\ar[d]&3\ar[r]\ar[d]&\cdots \ar[r] &n \ar[d] \\ 1'\ar[r]&2'\ar[r]&3'\ar[r]&\cdots \ar[r]&n'}
$$
with all commutative relations has Coxeter polynomial 
$$\chi_{2n}= \chi_{\A_n}\otimes \chi_{\A_2} = v_{n+1} \otimes v_3$$ 
\noindent
which is a product of cyclotomic polynomials, see \cite{KuLeMe}.

\subsection{Repetitive algebras}
Let $A$ be a triangular algebra with $n$ vertices. Consider the {\em double repetitive} algebra 
$$A^{(2)}=\begin{pmatrix}A&D(A)\\ 0 & A \end{pmatrix}$$ 
\noindent
making use of the $A-A$-bimodule structure of $D(A)$.

The Cartan matrix $C_A^{(2)}$, its inverse and the Coxeter matrix $\phi_A^{(2)}$ of $A^{(2)}$ are
{\footnotesize$$  C_A^{(2)}=\begin{pmatrix}C_A & C_A^t\\0&C_A \end{pmatrix}, \,\, 
(C_A^{(2)})^{-1}=\begin{pmatrix}C_A^{-1} & -C_A^{-1} C_A^t C_A^{-1} \\0&C_A^{-1}\end{pmatrix}{\rm \,\, and \,\,} 
\phi_A^{(2)}=\begin{pmatrix}0 & -{\rm id_n} \\ \phi_A^2 & \phi_A \end{pmatrix}$$}
\noindent
Consider an eigenvector $u$ of $\phi_A$ with eigenvalue $\lambda$. 
Let $a_1$ and $a_2$ be the roots of the equation $x^2-\lambda x +\lambda^2=0$. Therefore 
$a_m={1+(-1)^m \sqrt{-3} \over 2}\, \lambda$, for $m=1,2$. Then
{\footnotesize $$\phi_A^{(2)}\begin{pmatrix}u \\ - a_m u \end{pmatrix}= \begin{pmatrix} a_m u \\ \phi_A^2(u)- a_m \phi_A(u)\end{pmatrix}
= a_m \begin{pmatrix}u \\- a_m u \end{pmatrix}$$}
\noindent
Therefore the eigenvalues of $\phi_A^{(2)}$ are of the form ${1 \pm \sqrt{-3} \over 2}\, \lambda$ for eigenvalues
$\lambda$ of $\phi_A$.

\begin{propo}
If $A$ is an algebra of cyclotomic type then the double repetitive algebra is also of cyclotomic type with $\chi_{A^{(2)}}=\Phi_6 \otimes \chi_A$.
\end{propo}
\begin{proof}
Just observe that $|{1\pm \sqrt{-3} \over 2}|=1$.
\end{proof}

\subsection{Groups of automorphisms}
Let $A=kQ/I$ be a triangular algebra given as a quiver algebra $kQ$ with an ideal $I$, equipped with a group $G$ of automorphisms induced by automorphisms of $Q$. Each $g \in G$ induces a permutation matrix
$\gamma(g) \in {\rm GL}_n(\Z)$ which is called a {\em symmetry}. The representation $\gamma: G \to {\rm GL}_n(\Z)$ is called the {\em canonical representation}.
Clearly, $\phi_A$ is an automorphism of $\gamma$. We shall recall well-known results concerning the action of
$G$ on $A$.

Let $R_1,\ldots,R_q$ be a set of representatives of the {\it irreducible}
$\Q$-representations of $G$, where $R_1$ is the trivial representation.
By Maschk\'e's theorem, there exists an invertible rational matrix $L$ such that
the conjugate representation $\gamma^L$ has a decomposition
$$
\gamma^L = \bigoplus\limits_{\alpha =1}^q R_\alpha^{r(\alpha)}.
$$

The dimension of $R_{\alpha}$ is denoted by $\dim R_\alpha$, that is,
$R_\alpha : G\to \ {\rm GL}_{\dim R_\alpha}(\Z)$. Then we have $n =
\sum\limits_{\alpha =1}^q r(\alpha) \dim R_\alpha$.
Since $\phi_A$ is an automorphism of $\gamma$, by Schur's lemma, the
conjugate $\phi_A^L (:= L\phi_A L^{-1})$ takes the block diagonal
form
$$
\phi_A^L = \begin{pmatrix}\phi_1& & 0\cr & & \cr &\ddots&\cr & & \cr
0& &\phi_q\end{pmatrix}
$$
where $\phi_\alpha : R_\alpha^{r(\alpha)}\to R_\alpha^{r(\alpha)}$ is an
automorphism, that is, $\phi_\alpha \in {\rm GL}_{r(\alpha) \dim R_\alpha}(\Z)$.

Therefore we get a factorization
$$
\chi_A(T)=\chi_{\phi_1} (T) \ldots \chi_{\phi_q}(T)
$$
where $\chi_{\phi_\alpha} (T) \in \Z[T]$ is the characteristic polynomial of
$\phi_\alpha$. In particular, $\chi_A(T)$ has at least $\vert \{\alpha : r(\alpha) >
0 \}\vert$ factors.

Let $S_1,\ldots,S_m$ be a set of representatives of the
{\it irreducible} $\C$-representations of $G$. Let $S_1$ be the trivial
representation. Then $m$ is the number of conjugacy classes of $G$.

There is a conjugate $\gamma^M$ of $\gamma$ with a decomposition
$$
\gamma^{M} = \bigoplus\limits_{\beta =1}^m S_{\beta}^{n(\beta)}.
$$

Let $\chi_\beta$ be the {\it character} corresponding to $S_\beta$, that is
$\chi_\beta : G\to \C^*, g\mapsto tr S_\beta (g)$. The characters
$1=\chi_1,\ldots,\chi_m$ form an orthonormal basis of the class group
$X(G)$, with the scalar product
$$
\langle \chi,\chi' \rangle={1\over \vert G\vert} \sum\limits_{g\in G} \chi(g)\overline{\chi'(g)}.
$$
\noindent  
The following holds:

$\bullet$ For any $g\in G$, $\chi_\gamma (g)$ is the number of
fixed points of $g$ on the vertices of $A$.

$\bullet$\quad {\it Burnside's lemma:} the number of orbits of vertices of $A$ is
$$
t_0 (G)= {1\over \vert G\vert} \sum\limits_{g\in G} \chi_\gamma (g).
$$
\noindent
Moreover, $n(1)=t_0(G) = r(1)$.

\begin{propo}
If $A$ has non-trivial symmetries, then $\chi_A(T)$ is not irreducible over $\Z[T]$. Moreover the following holds:

{\rm (a)} $\chi_A(T)$ accepts a factor of degree $ \le t_0(G)$;

{\rm (b)} in case $A$ is of cyclotomic type, there is a factor $\Phi_m(T)$ of $\chi_A(T)$ with $m \le t_0(G)^2$. 
\end{propo}
\begin{proof}
Consider $e_1,\ldots,e_n$ the canonical basis of $K_0(A)=\Z^n$. Consider the subspace $V$ of $G$-invariant vectors 
in $K_0(A) \otimes \Q$. Then $\dim_\Q V =t_0(G)$.

{\rm (a):} For a decomposition $K_0(A) \otimes \Q= V \oplus V'$ the transformation $\phi_A$ takes the form
$$\begin{pmatrix}\phi_1 & * \\ 0 & \phi_2 \end{pmatrix}$$
\noindent
which yields $\chi_A(T)=\chi_{\phi_1}(T) \chi_{\phi_2}(T)$ with ${\rm degree}(\chi_{\phi_1}(T)) = t_0(G)$.

{\rm (b):} Assume that $\chi_A(T)= \prod_{m \in M} \Phi_m(T)^{e(m)}$ for some set $M \subset \N$ and numbers $e(m) \ge 1$. 
Then there is some $m_0 \in M$ satisfying
$$\phi(m_0) ={\rm degree}(\Phi_{m_0}(T)) \le t_0(G).$$
\noindent
By \cite{KeOs}, the totient function satisfies $\phi(m) \ge \sqrt{m}$, except 
for $m=2$ and $m=6$. We distinguish two situations:

{\rm (i):} Assume $t_0(G) \ge 3$. If $m_0 =2$ or $6$ then $m_0 \le t_0(G)^2$. Otherwise $m_0 \le \phi(m_0)^2 \le t_0(G)^2$.

{\rm (ii):} Assume $t_0(G) \le 2$. We observe that $t_0(G)$ is the number of orbits of vertices of $Q$ under the action of $G$.
Since $A$ is triangular, there are at least two orbits, one corresponding to sources and another corresponding to sinks in $Q$.
Therefore $t_0(G)=2$ and all sources belong to an orbit $Gi_0$ and all sinks to an orbit $Gj_0$.

In case $G$ fixes $i_0$ then $Q$ has the following shape
{\footnotesize
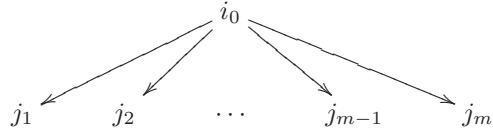
$$\xymatrix{&& i_0 \ar[dll] \ar[dl] \ar[dr] \ar[drr]&&\\ j_1 & j_2& \cdots&j_{m-1}& j_m}$$}
\noindent
with $G$ the group of permutations $S_m$ for $m \ge 2$. We observe that 
$$ \chi_A=(T^2 -(m-2)T+1)(T+1)^{m-1}$$
\noindent
which is a product of cyclotomic polynomials only if $m \le 4$. In this situation $\Phi_1$ is a factor of $\chi_A$.

In case $G$ acts freely on vertices of $Q$ then the quotient $A \to B=A/G$ under the action of $G$ is a Galois quotient in
the sense recalled in the next section. Then $B$ is the algebra of the quiver with two vertices $a$ and $b$ and $m \ge 2$ arrows
from $a$ to $b$. Its characteristic polynomial is 
$$\chi_B=T^2-(m^2-2)T+1$$
\noindent
We shall prove in  that, along with $\chi_A$, the polynomial $\chi_B$ is cyclotomic. This is possible only if $m=2$.
Therefore $G$ is cyclic and $Q$ has the shape
{\footnotesize
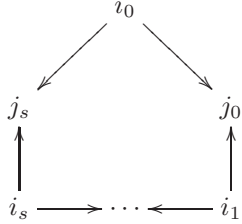
$$\xymatrix{&&i_0 \ar[ld]\ar[rd]&&\\&j_s&&j_0&\\&i_s \ar[u]\ar[r]&\cdots&i_1 \ar[u]\ar[l]&}$$}
\noindent
which has Coxeter polynomial 
$$\chi_A=(T-1)^2(T^s+T^{s-1}+\ldots+T+1)^2$$
\noindent
and accepts $\Phi_1$ as a factor. 
\end{proof}

\subsection{Galois quotients}
Let $A=kQ/I$ be an algebra given as a quiver algebra $kQ$ with an ideal $I$, equiped with a group $G$ of automorphisms induced by automorphisms of $Q$ acting without fixed vertices. Then the quotient $B=A/G$
is an algebra with a presentation $B=k{\overline Q}/{\overline I}$ where ${\overline Q}=Q/G$ is the quotient defined by the action of $G$ on the quiver $Q$. We say that $B$ is a {\em Galois quotient} of $A$ and we say that $A$ is a {\em Galois cover} of $B$ defined by the group $G$, we write $B=A/G$. We say that $B$ is simply connected if whenever $B=A/G$ is a Galois quotient then $G$ is trivial.

\begin{propo}
Let $B=A/G$ be a Galois quotient algebra of $A$. If $A$ is an algebra of cyclotomic type then $B$ is of cyclotomic type.
\end{propo}
\begin{proof}
Consider the Cartan matrix $C_A$ of $A$. The columns $p_i=e_iC_A$ of $C_A$, for $i=1,\ldots,n$, are the
dimension vectors of the indecomposable projective $A$-modules $P_i=A e_i$. Correspondingly, the indecomposable projective $B$-module ${\overline P}_s=B {\overline e}_s$ has dimension vector 
${\overline p}_s={\overline e}_s C_B$, where $s=1,\ldots,m$ runs among the classes of the action of $G$ on the vertices of $Q$.

Assume that $A$ is of cyclotomic type. Let $\lambda$ be an eigenvalue of $\phi_B=-C_B C_B^{-t}$ with associated eigenvector $u \in K_0(B)$. We shall prove that $\lambda$ is an eigenvalue of $\phi_A$ which implies that $\lambda$ is a root of unity.

Indeed, define $v \in K_0(A)$ a vector $v(i)=u(s)$, where $s=Gi$. Then
$$\lambda (v C_A^t)(i)=\lambda \sum\limits_{j=1}^n v(j) p_j(i)= 
\lambda \sum\limits_{t=1}^m u(t) \left( \sum\limits_{g \in G} p_{g(j)}(i) \right)=$$  
$$=\lambda \sum\limits_{t=1}^m u(t){\overline p}_t(s)= \lambda (u C_B^t)(s)= -(uC_B)(s)$$
\noindent
and $(uC_B)(s)=(vC_A)(i)$ by the same argument. Thus $\lambda vC_A^t= -v C_A$ as we wanted to prove. 
\end{proof}

\subsection{Canonical and supercanonical algebras}

A canonical algebra $\Lambda$ is determined by a weight sequence ${\bf p}=(p_1,\ldots,p_t)$ of $t$ integers $p_i\geq2$ and a parameter sequence ${\bf \lambda}=(\lambda_3,\ldots,\lambda_t)$ consisting of $t-2$ pairwise distinct non-zero scalars  from the base field $k$. (We may assume $\lambda=1$ such that for $t\le 3 $ no parameters occur).  Then the algebra $\Lambda$ is defined by the quiver
$$
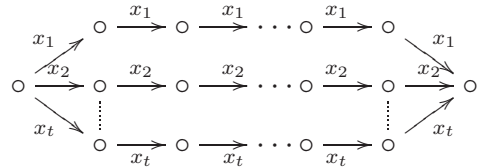
\def\c{\circ}
\xymatrix@C18pt@R10pt{
                                        &\c\ar[r]^{x_1}          & \c\ar[r]^{x_1}         &\cdots\c\ar[r]^{x_1}&\c\ar[dr]^{x_1}&\\
\c\ar[r]^{x_2}\ar[ru]^{x_1}\ar[rd]_{x_t}&\c\ar[r]^{x_2}\ar@{.}[d]&\c\ar[r]^{x_2}          &\cdots\c\ar[r]^{x_2}&\c\ar[r]^{x_2}\ar@{.}[d] &\c    \\
                                        &\c\ar[r]_{x_t}          &\c\ar[r]_{x_t}          &\cdots\c\ar[r]_{x_t}&\c\ar[ru]_{x_t}&    \\
}
$$
satisfying the $t-2$ equations:
$$x_i^{p_i}=x_1^{p_1}-\lambda_i x_2^{p_2},\qquad i=3,\ldots,t.$$
For more than two weights, canonical algebras are not hereditary. Instead, their representation theory is determined by a hereditary category, the category $\coh\X$ of coherent sheaves on a {\em weighted projective line} $\X$, naturally attached to $\Lambda$, see~\cite{GeLe}.
\begin{propo}
 Let $\Lambda$ be a canonical algebra. Then $\Lambda$ is the endomorphism ring of a tilting object in the category $\coh\X$ of coherent sheaves on the weighted projective line $\X$. The category $\coh\X$ is hereditary and satisfies Serre duality in the form $D \Ext^1(X,Y)=\Hom(Y,{\tau X})$ for a self-equivalence $\tau$ which serves as the Auslander-Reiten translation.~\hfill\qed
\end{propo}
Canonical algebras were introduced by Ringel~\cite{Ri}. They form a key class to study important features of representation theory. In the form of tubular canonical algebras they provide the standard examples of tame algebras of linear growth. Up to tilting canonical algebras are characterized as the connected $k$-algebras with a separating exact subcategory or a separating tubular one-parameter family (see \cite{LePe1} and \cite{Sk2}).
That is, the module category $\mod\Lambda$ accepts a {\em separating tubular family}
$\mathcal{T}=(T_\lambda)_{\lambda \in
P_1k}$, where $T_\lambda$ is a homogeneous tube for all $\lambda$
with the exception of $t$ tubes $T_{\lambda_1},\ldots,T_{\lambda_t}$
with $T_{\lambda_i}$ of rank $p_i$ ($1\le i\le t$). The Coxeter polynomial of $\Lambda$ is given by
$$
\chi_{\Lambda}=(x-1)^2\prod_{i=1}^t v_{p_i}.
$$
\noindent
In particular, canonical algebras are of cyclotomic type.

Following \cite{LePe2},
a supercanonical algebra is defined as follows:
The \emph{double cone} $\hat{S}$ of a finite poset $S$ is the poset obtained from $S$ by adjoining a unique source $\alpha$ and a unique sink $\omega$, like in the pictured example
$$
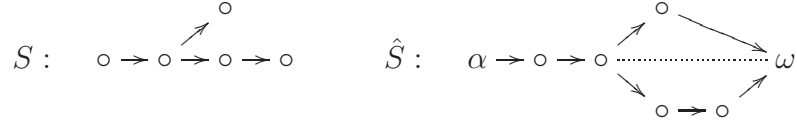
\def\c{\circ}
\def\b{\bullet}
\xymatrix@C10pt@R4pt{
   &         &                &\c      &  & &        &        &        &                &\c\ar[rrd]&         &  \\
S: & \c\ar[r]&\c\ar[ru]\ar[r] &\c\ar[r]&\c& &\hat{S}: &\alpha\ar[r]&\c\ar[r]&\c\ar[ur]\ar[dr]\ar@{.}[rrr]&          &         &\omega\\
   &         &                &        &  & &        &        &        &                &\c\ar[r]  &\c\ar[ur]&  \\
}
$$
Due to commutativity there is a unique path $\kappa_S$ from $\alpha$ to $\omega$ in $\hat{S}$.
Let now $S_1,\ldots,S_t$ be  finite posets, $t\ge 2$ and $\lambda_3,\ldots,\lambda_t$
pairwise different elements from  $ k\setminus \{0\}$. The supercanonical algebra
$A=A(S_1,\ldots,S_t;\lambda_3,\ldots,\lambda_t)$ is obtained from the fully commutative quivers $\hat{S}_1,\ldots,\hat{S}_t$ by identification of the sources and sinks, respectively, and requesting additionally the
$t-2$ relations $\kappa_i=\kappa_2-\lambda_i\kappa_1$, $i=3,\ldots,t$, where $\kappa_i=\kappa_{S_i}$. The next figure yields an example of a supercanonical algebra with three arms
$$
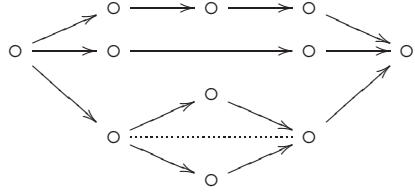
\def\c{\circ}
\xymatrix@R4pt{
                           &\c\ar[r]          &\c\ar[r]          &\c\ar[rd]        &  \\
\c\ar[r]\ar[ur]\ar[ddr]    &\c\ar[rr]          &                  &\c\ar[r]         &\c\\
                           &                  &\c\ar[dr]         &                 &  \\
                           &\c\ar[ur]\ar[dr]\ar@{.}[rr]&                  &\c\ar[uur]       &  \\
                           &                  &\c\ar[ur]         &                 &  \\
}
$$
where we assume that $\kappa_3=\kappa_2-\kappa_1$.
In case, $S_1,\ldots,S_t$ are linear quivers
$S_i=[p_i]\colon 1\to 2\to \cdots \to p_{i-1}$,
the algebra $A(S_1,\ldots,S_t;\lambda_1,\ldots,\lambda_t)$
 is just the canonical algebra $\Lambda(p_1,\ldots,p_t;\lambda_3,\ldots,\lambda_t)$.

Returning to the general case, a simple calculation shows that
$$\chi_{A}=(x-1)^2\prod^t_{i=1}\chi_{S_i}$$
where $\chi_{S_i}$ is the Coxeter polynomial of the poset algebra
$S_i$, $i=1,\ldots,t$. In particular, in case every poset algebra $k[S_i]$ is of cyclotomic type, for $i=1,\ldots,t$
then the supercanonical algebra $A$ is of cyclotomic type.

\subsection{One-point extensions of canonical algebras} \label{ssect:onepoint_canonical}
For detailed results related to one-point extensions by exceptional modules over canonical algebras we refer to \cite{LePe3, LePe4}.

Let $\Lambda$ be a canonical algebra of weight type $(p_1,\ldots,p_t)$.

$(a)$ If $M$ is regular simple in the $i$-th exceptional tube $\T_i$ (of $\tau$-period $p_i$), then the one-point extension $\Lambda[M]$ is tilting-equivalent to the canonical algebra of weight type $(p_1,\ldots,p_i+1,\ldots,p_t)$ having the same parameter sequence as $\Lambda$.

$(b)$ If $M$ has quasi-length $s$ in $\T_i$ (recall this means that $s<p_i$), then $\Lambda[M]$ is derived equivalent to a supercanonical algebra in the sense of \cite{LePe2}, where the linear arms of the canonical algebra with index different from $i$ are kept, and the $i$-th linear arm is changed to the poset $K(p_i,s)$
\footnotesize$$
\xymatrix@R8pt@C14pt{
       &       &             &                   & \star         &         &    \\
1\ar[r]&2\ar[r]&\cdots\ar[r] &p_i-s\ar[r]\ar[ru] & p_i-s+1\ar[r] &\cdots\ar[r]&p_i-1
}
$$\normalsize
We write $\Lambda(i,s)$ for $\Lambda[M]$ and call it a {\em supercanonical algebra of restricted type}.

The structure of the bounded derived category of an extended canonical algebra
$\langle{p_1,\dots,p_t}\rangle$ sensibly depends on the sign of the Euler
characteristic $\chi_{\x}=2-\sum_{i=1}^t(1-1/p_i)$ of the weighted projective line $\x$
associated to $\Lambda$. According to \cite{LePe4}, the description of the derived category of an extended canonical algebra yields an interesting trichotomy.

$(i)$ Positive Euler characteristic: Let $\Lambda$ be a canonical algebra of domestic representation type $(p_1,p_2,p_3)$, and $\Delta$ be the Dynkin diagram $[p_1,p_2,p_3]$. Then the extended canonical algebra $\langle{p_1,p_2,p_3} \rangle$ is derived equivalent to the (wild) path algebra  of a quiver $Q$ having extended Dynkin type $\tilde{\tilde{\Delta}}$.

$(ii)$ Euler characteristic zero: Consider a canonical algebra with a tubular weight sequence $(p_1,\ldots,p_t)$, we shall assume that $2\le p_1\le p_2\le \cdots \le p_t$. Then the extended canonical algebra
$\langle {p_1,\dots,p_t} \rangle$ is derived canonical of type $(p_1,\ldots,p_{t-1},p_t+1)$.

$(iii)$ Negative Euler characteristic: Let $\x$ be a weighted projective line of negative Euler characteristic, let $A= \langle {p_1,\dots,p_t} \rangle$ be the corresponding extended canonical algebra and $R$ be the $\Z$-graded singularity attached to $\x$. Then the derived category $\Der(\mod A)$ is triangle-equivalent to the triangulated category of graded singularities $\Der_{{\rm sing}}^{{\Z}}(R)$ of $R$, where the superscript $\Z$ refers to the grading.

\subsection{Extended canonical algebras of critical type.}

In \cite{LePe4} the extended canonical algebras $A$ with minimal weigth type $(p_1,p_2,\ldots, p_t)$ such that
${\rm Spec \,}(\chi_A(T)) \not\subset \s^1$ were classified. Recall that such an extended canonical algebra satisfies
$\dim_k H^2(A)= t-3$. In particular, if $A$ is accessible then $t=3$. We get the following more precise statement.

\begin{theor}
Let $A$ be an algebra derived equivalent to an extended canonical algebra of
weight ${\bf p}= (p_1,p_2,\ldots,p_t)$. Suppose that $A$ is not derived equivalent to a canonical algebra but $A$ is of cyclotomic type then the following holds:

(a) $A$ belongs to the {\rm Table} below;

(b) suppose that  $B$ is an algebra derived equivalent to an extended canonical algebra of weight ${\bf q}= (q_1,q_2,\ldots,q_t)$ dominated by ${\bf p}$ then $B$ is of cyclotomic type;

(c) $A$ is accessible if and only if $t=3$;

(d) for $t=3$ the algebra $A$ is Calabi-Yau.
\end{theor}
\medskip
Statements (a), (b) and (c) are shown in \cite{LePe4}, as explained in the next paragraph. 
Statement (d) is shown in \cite{LePe6}.

\subsection{} 
Write $A=C[P]$ for a wild canonical algebra $C$, whose module category is derived equivalent to $\coh \, \X$ for a weighted projective line $\X$, 
and $P$ an indecomposable projective $C$-module. Then
$\chi_A=P_C \chi_C$, 
where $P_C$ is the Hilbert-Poincar\'e series of the positively graded algebra
$R(p,\lambda)=\bigoplus\limits^\infty_{n=0}\Hom_C(M,\tau^n_CM)$
and $M$ is a rank one not preprojective $C$-module. Equivalently,
$R(p,\lambda)=\bigoplus\limits^\infty_{n=0}\Hom(\Oo,\tau_\X^n\Oo)$,
where $\Oo$ is the structure sheaf on $\X$. Recall from \cite{Le1}
that in case $k=\C$, we can interpret $R(p,\lambda)$ as an algebra
of entire automorphic forms associated to the action of a suitable
{\em Fuchsian group\/} of the first kind, acting on the upper half
plane $\h_+$.

It is well-known that the $k$-algebra $R=R(p,\lambda)$ is
commutative, graded integral Gorenstein, in particular
Cohen-Macaulay, of Krull dimension two. The complexity of the
surface singularity $R$ is described by the triangulated category
$$
\Der_{\rm sing}^\Z(R)=\frac{\Der^b(\mod^\Z R)}{\Der^b({\rm proj}^\Z R)},
$$
where $\mod^\Z R$ (resp.\ ${\rm proj}^\Z R$)  denotes the
category of finitely generated (resp.\ finitely generated
projective) $\Z$-graded $R$-modules. This category was considered
by Buchweitz \cite{Bu} and Orlov \cite{Or1, Or2}.  It was shown in \cite{LePe4} 
that there exists a tilting object $\bar{T}$ in the
triangulated category $\Der_{\rm sing}^\Z(R)$ whose endomorphism ring is
isomorphic to an extended canonical algebra $C[P]$.

We say that the algebra $R(p,\lambda)$ (and also the weight sequence $p$) is {\em formally $n$-generated\/} if
$$P_C(T)=\frac{\prod\limits^{n-2}_{i=1}(1-T^{c_i})}{\prod\limits^n_{j=1}(1-T^{d_j})}$$
\noindent
for certain natural numbers $c_1,\ldots,c_{n-2}$ and $d_1,\ldots,d_n$, all $\ge 2$. The
algebra $R(p,\lambda)$ (and also the weight sequence $p$) is {\em
formally a complete intersection\/} if $P_C(T)$ is a rational
function $f_1(T)/f_2(T)$, where each $f_i(T)$ is a product of
cyclotomic polynomials. In \cite{LePe4} it was shown  that the following statements are equivalent:

\begin{itemize}
\item[{\rm (a)}] $R(p,\lambda)$ is formally $3$- or $4$-generated;

\item[{\rm (b)}] $R(p,\lambda)$ is formally a complete intersection;

\item[{\rm (c)}] $A$ is of cyclotomic type.
\end{itemize}

Moreover, for $t=3$ the algebra $R(p,\lambda)$ is a graded complete
intersection of the form $k[X_1,\ldots,X_s]/(\rho_3,\ldots,\rho_s)$
where $s=3$ or $4$ and $\rho_3,\ldots,\rho_s$ is a homogeneous
regular sequence. For $k=\C$ the assertion also holds for $t\geq4$
for $R(p,\lambda')$ for a suitable choice of parameters.
$\lambda'=(\lambda'_3,\ldots,\lambda'_t)$.

In case the algebra $R=R(p,\lambda)$ is formally $3$-generated and $t=3$ then $R$ is always a graded complete
intersection of the form $k[X_1,X_2,X_3]/(f)$. In a forthcoming publication \cite{LePe6} we show that the singularity category 
$\Der_{\rm sing}^\Z(R)$ satisfies the CY property, in particular showing statement (d) of the above theorem.

In {\rm Table\/}, the marks $\bulitito$ and
\cuadrito\ refer to the case $k=\C$: those weight sequences marked
by $\bulitito$ or \cuadrito\ correspond to algebras $R(p,\lambda)$
associated to hypersurface singularities, in those cases
$R(p,\lambda)$ is formally $3$-generated. The marks $\bulitito$
correspond to {\em Arnold's $14$ exceptional unimodal
singularities\/} in the theory or singularities of differentiable
maps \cite{Le1}. Among those weight sequences $p=(p_1,p_2,p_3)$ (that
is $t=3$), Arnold's singularities are exactly those rings of
automorphic forms having three generators.
\newpage

\begin{center}
{\tiny
\begin{tabular}{rc|l|l|c}
&Weight sequence\ \ &\ \ Factorization of $\chi_A$\ \ &\ \
Poincar\'e
  series\ \ &\ \ Period of $\phi_A$\ \ \\[3.5pt]
\hline
&&&&\\[-9pt]
$\bulitito$ &$(2,3,7)$ &$\Phi_{42}$ &$(6,14,21)$ $(42)$ &$42$\\[3.5pt]
$\bulitito$ &$(2,3,8)$ &$\Phi_2\cdot \Phi_{10}\cdot \Phi_{30}$
  &$(6,8,15)$ $(30)$ &$30$\\[3.5pt]
$\bulitito$ &$(2,3,9)$ &$\Phi_3\cdot \Phi_{12}\cdot \Phi_{24}$
  &$(6,8,9)$ $(24)$ &$24$\\[3.5pt]
&$(2,3,10)$ &$\Phi_2\cdot \Phi_{16}\cdot \Phi_{18}$
  &$(6,8,9,10)$ $(16,18)$ &$72$\\
\multicolumn{5}{c}{\phantom{\tiny .}\punteada\phantom{\tiny .}}\\
$\bulitito$ &$(2,4,5)$ &$\Phi_2\cdot \Phi_6\cdot \Phi_{30}$
  &$(4,10,15)$ $(30)$ &$30$\\[3.5pt]
$\bulitito$ &$(2,4,6)$ &$\Phi^2_2\cdot \Phi_{22}$
  &$(4,6,11)$ $(22)$ &$22$\\[3.5pt]
$\bulitito$ &$(2,4,7)$ &$\Phi_2\cdot \Phi_9\cdot \Phi_{18}$
  &$(4,6,7)$ $(18)$ &$18$\\[3.5pt]
&$(2,4,8)$ &$\Phi^2_2\cdot \Phi_4\cdot \Phi_{12}\cdot \Phi_{14}$
  &$(4,6,7,8)$ $(12,14)$ &$84$\\
\multicolumn{5}{c}{\phantom{\tiny .}\punteada\phantom{\tiny .}}\\
$\bulitito$ &$(2,5,5)$ &$\Phi_5\cdot \Phi_{20}$
  &$(4,5,10)$ $(20)$ &$20$\\[3.5pt]
$\bulitito$ &$(2,5,6)$ &$\Phi_2\cdot \Phi_8\cdot \Phi_{16}$
  &$(4,5,6)$ $(16)$ &$16$\\[3.5pt]
&$(2,5,7)$ &$\Phi_{11}\cdot \Phi_{12}$
  &$(4,5,6,7)$ $(11,12)$ &$132$\\[3.5pt]
&$(2,6,6)$ &$\Phi^2_2\cdot \Phi_3\cdot \Phi_6\cdot \Phi_{10}\cdot \Phi_{12}$
  &$(4,5,6,6)$ $(10,12)$ &$60$\\
\multicolumn{5}{c}{\phantom{\tiny .}\punteada\phantom{\tiny .}}\\
$\bulitito$ &$(3,3,4)$ &$\Phi_3\cdot \Phi_{24}$
  &$(3,8,12)$ $(24)$ &$24$\\[3.5pt]
$\bulitito$ &$(3,3,5)$ &$\Phi_2\cdot\Phi_3\cdot\Phi_6\cdot\Phi_{18}$
  &$(3,5,9)$ $(18)$ &$18$\\[3.5pt]
$\bulitito$ &$(3,3,6)$ &$\Phi^2_3\cdot\Phi_{15}$
  &$(3,5,6)$ $(15)$ &$15$\\[3.5pt]
&$(3,3,7)$ &$\Phi_2\cdot\Phi_3\cdot\Phi_4\cdot\Phi_{10}\cdot\Phi_{12}$
  &$(3,5,6,7)$ $(10,12)$ &$60$\\[3.5pt]
$\bulitito$ &$(3,4,4)$ &$\Phi_2\cdot\Phi_4\cdot\Phi_{16}$
  &$(3,4,8)$ $(16)$ &$16$\\[3.5pt]
$\bulitito$ &$(3,4,5)$ &$\Phi_{13}$ &$(3,4,5)$ $(13)$ &$13$\\[3.5pt]
&$(3,4,6)$ &$\Phi_2\cdot\Phi_3\cdot\Phi_9\cdot\Phi_{10}$
  &$(3,4,5,6)$ $(9,10)$ &$90$\\[3.5pt]
&$(3,5,5)$ &$\Phi_2\cdot\Phi_5\cdot\Phi_8\cdot\Phi_{10}$
  &$(3,4,5,5)$ $(8,10)$ &$40$\\[3.5pt]
$\bulitito$ &$(4,4,4)$ &$\Phi^2_2\cdot\Phi^2_4\cdot\Phi_6\cdot\Phi_{12}$
  &$(3,4,4)$ $(12)$ &$12$\\[3.5pt]
&$(4,4,5)$ &$\Phi_2\cdot\Phi_4\cdot\Phi_8\cdot\Phi_9$
  &$(3,4,4,5)$ $(8,9)$ &$72$\\
\multicolumn{5}{c}{\phantom{\tiny .}\punteada\phantom{\tiny .}}\\
$\cuadrito$ &$(2,2,2,3)$ &$\Phi^2_2\cdot \Phi_{18}$
  &$(2,6,9)$ $(18)$ &$18$\\[3.5pt]
$\cuadrito$ &$(2,2,2,4)$ &$\Phi^2_2\cdot \Phi_{14}$
  &$(2,4,7)$ $(14)$ &$14$\\[3.5pt]
$\cuadrito$ &$(2,2,2,5)$ &$\Phi^2_2\cdot\Phi_3\cdot\Phi_6\cdot \Phi_{12}$
  &$(2,4,5)$ $(12)$ &$12$\\[3.5pt]
&$(2,2,2,6)$ &$\Phi^2_2\cdot \Phi_8\cdot \Phi_{10}$
  &$(2,4,5,6)$ $(8,10)$ &$40$\\[3.5pt]
$\cuadrito$ &$(2,2,3,3)$ &$\Phi_2\cdot \Phi_3\cdot\Phi_4\cdot \Phi_{12}$
  &$(2,3,6)$ $(12)$ &$12$\\[3.5pt]
$\cuadrito$ &$(2,2,3,4)$ &$\Phi^2_2\cdot\Phi_5\cdot \Phi_{10}$
  &$(2,3,4)$ $(10)$ &$10$\\[3.5pt]
&$(2,2,3,5)$ &$\Phi_2\cdot \Phi_7\cdot \Phi_8$
  &$(2,3,4,5)$ $(7,8)$ &$56$\\[3.5pt]
&$(2,2,4,4)$ &$\Phi^2_2\cdot \Phi_4\cdot \Phi_6\cdot \Phi_8$
  &$(2,3,4,4)$ $(6,8)$ &$24$\\[3.5pt]
$\cuadrito$ &$(2,3,3,3)$ &$\Phi^2_3\cdot \Phi_9$
  &$(2,3,3)$ $(9)$ &$9$\\[3.5pt]
&$(2,3,3,4)$ &$\Phi_2\cdot \Phi_3\cdot \Phi_6\cdot\Phi_7$
  &$(2,3,3,4)$ $(6,7)$ &$42$\\[3.5pt]
&$\xy (0,0) *+[F]{(3,3,3,3)}\endxy$ &$\Phi_2\cdot\Phi^3_3\cdot\Phi^2_6$
  &$(2,3,3,3)$ $(6,6)$ &$\infty$\\[3.5pt]
$\cuadrito$ &$(2,2,2,2,2)$ &$\Phi^4_2\cdot \Phi_{10}$
  &$(2,2,5)$ $(10)$ &$10$\\[3.5pt]
$\cuadrito$ &$(2,2,2,2,3)$ &$\Phi^3_2\cdot \Phi_4\cdot \Phi_8$
  &$(2,2,3)$ $(8)$ &$8$\\[3.5pt]
&$\xy (0,0) *+[F]{(2,2,2,2,4)}\endxy$ &$\Phi^2_2\cdot \Phi_3\cdot\Phi^2_6$
  &$(2,2,3,4)$ $(6,6)$ &$\infty$\\[3.5pt]
&$(2,2,2,3,3)$ &$\Phi^2_2\cdot \Phi_3\cdot\Phi_5\cdot\Phi_6$
  &$(2,2,3,3)$ $(5,6)$ &$30$\\[3.5pt]
&$(2,2,2,2,2,2)$ &$\Phi^5_2\cdot \Phi_4\cdot\Phi_6$
  &$(2,2,2,3)$ $(4,6)$ &$12$
\end{tabular}
}
\vskip12pt
{\bf Table} Weights $p$ with $\rho (\phi_A)=1$.
\end{center}

\section{Algebras of cyclotomic type and the homological quadratic form} 

\subsection{}
Let $A$ be an algebra of cyclotomic type. As observed in \cite{La,LePe3} the transformation $\phi_A$ is not necessarily periodic. 
As we remarked before, $\phi_A$ is periodic if and only if $A$ is of cyclotomic type and $\phi_A$ is diagonalizable.

The following result follows \cite{Ho}, see also \cite{LePe3, Sa}.

\begin{theor}
If $h_A$ is non-negative, then $A$ is of cyclotomic type. Moreover, if $h_A$ is positive definite then 
$\phi_A$ is diagonalizable and periodic.
\end{theor}
\begin{proof}
Let $C$ be any invertible positive upper triangular matrix and 
$h(x)=x^t(C^{-1}+C^{-t})x$, for $x \in K_0(A)$, an associated quadratic form. Assume
$h(x)$ is positive definite, we shall prove that all eigenvalues of $-C C^{-t}$ have modulus one.

Indeed, if $C^{-t}u=-\lambda C^{-1}u$ for an eigenvalue $\lambda$ of $-C C^{-t}$ and a complex vector $u \ne 0$, since $0 < h(u) = {\overline u}^t C^{-t}u = - \lambda {\overline u}^t C^{-1}u$ then
$|| \lambda||=1$. In particular, this shows that $h_A$ positive definite implies $A$ cyclotomic.

Assume that $h_A$ is non-negative. Choose $\epsilon > 0$ such that 
$$C_\epsilon =C_A \sum\limits_{i=0}^\infty (-1)^i (\epsilon C_A)^i$$ 
\noindent
converges. Therefore $C_\epsilon^{-1}=C_A^{-1}(1 +\epsilon C_A)$ and the real quadratic form
$$h_\epsilon(x)=x^t C_\epsilon^{-1}x = x^t C_A^{-1}x + \epsilon x^t x$$
\noindent
is positive definite.

Since $\lim_{\epsilon \to 0}h_\epsilon = h_A$ and the eigenvalues of $h_\epsilon$ have modulus one, for every $\epsilon > 0$, then the same holds for the eigenvalues of $h_A$. That is, $A$ is of cyclotomic type.
 
Assume that $h_A$ is positive definite. Following \cite{How} we show that $\phi_A = -C_A C_A^{-t}$ is diagonalizable. Indeed, suppose, to get a contradiction, that $v$ is an eigenvector of $\phi_A$ with eigenvalue $\lambda$ and $0 \ne u$ is a vector with $\phi_A u - \lambda u = v$. We get

(1) $- C_A^{-t}u -\lambda C_A^{-1}u = C_A^{-1}v$ 

(2) $-C_A^{-t}v = \lambda C_A^{-1}v$.

Therefore, from (2), ${\overline u}^t C_A^{-t}v = - \lambda {\overline u}^t C_A^{-1}v$. Taking transpose and conjugate, this yields, 
${\overline v}^t C_A^{-1}u = - {\overline \lambda} {\overline v}^t C_A^{-t}u$. Since $||\lambda||=1$,
equation (1) implies $h_A(v)= {\overline v}^t C_A^{-t}v =0$, contradicting the positivity of $h_A$.

Finally, $\phi_A$ is equivalent to a diagonal matrix $D$ whose diagonal entries are roots of unity. Thus $\phi_A$ is periodic. 
\end{proof}

\subsection{}
In certain cases the non-negativity of $h_A$ is related to the representation type of $A$, see for example \cite{Pe2} and the next paragraph. On the other hand, there are pairs of algebras $A$ and $B$ with $h_A=h_B$ and such that $A$ is representation tame while $B$ is representation wild.
We show a sequence of such pairs. 

For each number $n \ge 7$, consider the quiver
$$\footnotesize
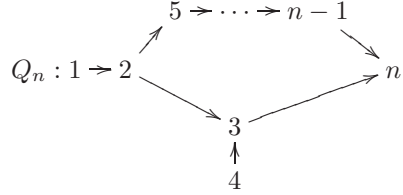
{\xymatrix@R8pt@C8pt{&&5 \ar[r]&\cdots \ar[r]&n-1\ar[rd]&\\Q_n: 1 \ar[r]& 2 \ar[ru] \ar[rrd]&&&&n\\&&&3 \ar[rru]&&\\&&&4 \ar[u]&&}}$$
\noindent
Let $A_n$ be the quotient of the path algebra $kQ_n$ determined by the commutativity relation and $B_n$ the quotient of 
$kQ_n$ where the long path of length $n-4$ from $2$ to $n$ vanishes. Then $h_{A_n}=h_{B_n}$ which is a positive definite
quadratic form. Observe that $A_n$ is representation finite while $B_n$ is a wild algebra.

\subsection{}
Consider the poset

$$\footnotesize
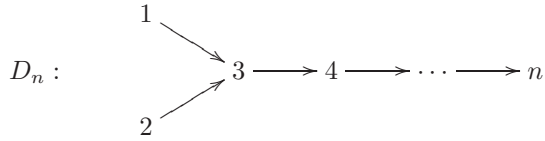
{\xymatrix@R8pt{&1 \ar[rd]&&&&\\D_n:  &&3 \ar[r]&4 \ar[r]&\cdots \ar[r]& n \\ &2 \ar[ru]&&&&}}$$

The supercanonical algebra $A=A((1),(1),D_n;1)$ is a {\em pg-critical algebra}, that is a tame algebra which is minimal 
not of polynomial growth. Any pg-critical algebra is derived equivalent to an algebra of the form
$A((1),(1),D_n;1)$, see \cite{LePe1}.
\medskip

A {\em semichain poset} is a full convex subcategory of the poset
$$\footnotesize
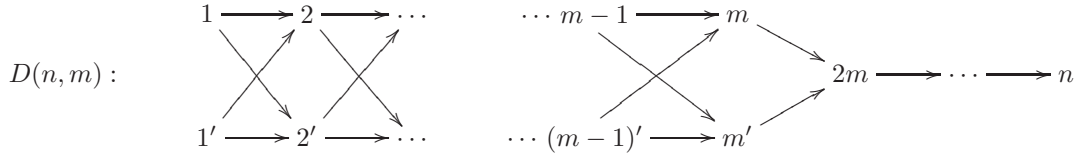
{\xymatrix@R8pt{&1 \ar[rdd]\ar[r]&2 \ar[rdd]\ar[r]& \cdots&\cdots \, m-1\ar[r]\ar[rdd] &m \ar[rd]&&&\\D(n,m):  &&&&&&2m \ar[r]&\cdots \ar[r]&n  \\ &1' \ar[r]\ar[ruu]&2' \ar[ruu]\ar[r]&\cdots&\cdots \, (m-1)'\ar[r]\ar[ruu] &m' \ar[ru] &&&}}$$

In important cases it is easy to determine the periodicity of the Coxeter transformation. Recall
that a supercanonical algebra $A=A(S_1,\ldots,S_t;\,\lambda_3,\ldots,\lambda_t)$ is said to be
{\em of Dynkin class} if each poset $S_i$ is of Dynkin type. In that case, let $p(S_i)$ be the period of $\phi_{S_i}$. 

\begin{theor}
Let $A=A(S_1,\ldots,S_t;\,\lambda_3,\ldots,\lambda_t)$ be a
supercanonical algebra of Dynkin class. Let $p= lcm(p(S_1),\ldots,p(S_t))$ and $B=A/(\alpha)$ be the algebra obtained by deleting the source 
$\alpha$.

The following are equivalent:

\begin{itemize}
\item[{\rm (a)}] $\phi_A$ is periodic;

\item[{\rm (b)}] $\phi_A^p=1$ and $p$ is the exact period of $\phi_A$;

\item[{\rm (c)}] $h_A$ is non-negative and $\rad \,h_A$ has rank two;

\item[{\rm (d)}] $h_B$ is non-negative and $\rad \, h_B$ has rank one.
\end{itemize}
\end{theor}

\section{On the structure of Auslander-Reiten components of algebras of cyclotomic type}

\subsection{}
Let $A$ be an  algebra of cyclotomic type and $\chi_A=\prod_{m \in M}\Phi_m^{e(m)}$ be an irreducible decomposition of its Coxeter polynomial.

Consider the bounded derived category $\Der(A)=\Der^b(\mod A)$ and let $\Gamma= \Gamma_{\Der(A)}$ be the corresponding Auslander-Reiten quiver. Recall that for any class $[X] \in K_0(\Der(A))$ and $X \to Y \to Z \to T(X)$ an Auslander-Reiten triangle then $[T(X)]= \phi_A([X])$.

\begin{theor}
Every component of $\Gamma$ is either a tube $\Z \A_\infty / (p)$ of finite period $p$ where $p=\phi(m)$ for some $m \in M$
or of the form $\Z \Delta$ for $\Delta$ a Dynkin or extended Dynkin diagram or of one of the shapes $\A_\infty$, $\A_\infty^\infty$ or $\D_\infty$.
\end{theor}
\begin{proof}
By \cite{Zh}, every component is either a tube or of the shape $\Z \Delta$ for some infinite quiver $\Delta$ 
which is locally finite and without oriented cycles. 

Consider the case of a component $\mathcal C$ of shape $\Z\Delta$ in $\Gamma$. We will show that a growth argument restricts the possible shapes of $\Delta$. We need some preparation. Fix a vertex $x_0$ in $\mathcal C$ with corresponding class $X_0 \in K_0(\Der(A))$. For each $n \in \N$ the Auslander-Reiten translation $x_n= T^n(x_0)$ with corresponding class $X_n \in K_0(\Der(A))$.
Consider the function
$$c(n)=\sum\limits_{i \in Q_0} [X_n](i)$$
\noindent 
which is the restriction of an additive function on the subcategory ${\mathcal C}(x_0)$ of $\mathcal C$ formed by the predecessors of $x_0$. Consider the complete section $\overline \Delta$ with $x_0$ as unique sink and $c_0: {\overline \Delta} \to \Z$ the vector obtained as restriction of  the function $c$ to $\overline \Delta$. Then $c(n)=\phi_{\overline \Delta}^n(c_0)(x_0)$.

In case $\Delta$ is not one of the indicated graphs, the growth of $(\phi_{\overline \Delta}^n)_{n \in \N}$ is exponential, as it is shown in 
\cite{Zh}. This completes the claim on the shape of $\Delta$.

Let $\mathcal C$ be a tubular component of shape $\Z\A_\infty /(p)$ in $\Gamma$. Let $X_1,X_2,\ldots,X_p$ be objects on the mouth of $\mathcal C$. Let $V$ be the subspace of $K_0(\Der(A)) \otimes \C$ generated by the $X_i$. Then $V$ is a $\phi_A$-invariant subspace. Assume  $K_0(\Der(A)) \otimes \C= V \oplus W$ then $\phi_A$ takes the form
$$\begin{pmatrix} \phi_1 & 0\\ * &  \phi_2 \end{pmatrix}$$
\noindent
Then $\chi_A=\chi_{\phi_1}\chi_{\phi_2}$ and since $\phi_1$ is an irreducible matrix then
$\chi_{\phi_1}$ is an irreducible factor of $\chi_A$. That is $\chi_{\phi_1}=\Phi_m$ for some $m \in M$. Then $p=\phi(m)$.
\end{proof}

\subsection{}
Let $A$ be a triangular algebra and $\Gamma$ the Auslander-Reiten quiver of its derived category $\Der(A)$. The {\em class quiver} $[\Gamma]$ of $A$ is formed by the quotients $[{\mathcal C}]$ of components $\mathcal C$ of $\Gamma$ obtained by identifying $X, Y \in {\mathcal C}$ if $[X]=[Y]$ in
$K_0(\Der(A))$. Observe that, according to \cite{HPR}, $[{\mathcal C}]={\mathcal C}$ for any component of $\Gamma$ with the possible exception of $[\Z\A_\infty]=\Z \A_\infty/(p)$ for some $p$.

\begin{theor}
Let $A$ be a triangular algebra and $\Gamma$ the Auslander-Reiten quiver of its derived category $\Der(A)$. Then the following holds: 

(a) if $A$ satisfies the CY property then $\phi_A$ is periodic;

(b) every component of $[\Gamma]$ is a tube if and only if  $\phi_A$ is periodic. 

If (b) holds then,

(c) $A$ is of cyclotomic type, and

(d) if $p_1,\ldots,p_s$ are the periods $> 1$ of non-homogeneous tubes in $[\Gamma]$ then $\sum\limits_{i=1}^s p_i \le n$.
\end{theor}
\begin{proof}
Let $\mathcal C$ be a tubular component of shape $\Z\A_\infty /(p)$ in $[\Gamma]$. As above we may prove that there is an irreducible factor $f_p(T)$ of $\chi_A(T)$ of degree $p$.

Assume that every component of $[\Gamma]$ is a tube. Consider a maximal set $X_1,\ldots,X_s$ in $\Der(A)$ satisfying that 

(i) $[X_i]$ lies on the mouth of a tube ${\mathcal T}_i$ in $[\Gamma]$ with period $p_i$ and  $f_{p_i}(T)$ is an irreducible factor of $\chi_A(T)$ of degree $p_i$;

(ii) $[X_j]$ does not lie on ${\mathcal T}_i$ for $j \ne i$;

(iii) $\prod_{i=1}^s f_{p_i}$ is a factor of $\chi_A(T)$.

Clearly, $\sum\limits_{i=1}^s p_i =n$. In particular, $\{\phi_A^m([X_i]): 1 \le i \le s, 0 \le m < p_i\}$ is a basis of $K_0(\Der(A))$. Therefore $\phi_A$ is periodic which implies that $A$ is of cyclotomic type.

Conversely, assume that $\phi_A$ is periodic. 
Let $\mathcal C$ be a component of shape $\Z\Delta$ in $\Gamma$. Observe that the classes $[X]$ 
for $X \in {\mathcal C}$ are periodic. Hence by \cite{HPR}, $[{\mathcal C}]$ is a tube, that is, $\Delta$ has the shape $\Z \A_\infty$.
\end{proof}

{\footnotesize
\noindent
Centro de Investigaci\'on en Matem\'aticas, A.C.\\
Guanajuato 36240 M\'exico\\
\hbox{\hskip.18cm} {\footnotesize\em e-mail:\/} jap@cimat.mx \\
and\\
Instituto de Matem\'aticas, UNAM. Cd. Universitaria, \\
M\'exico 04510 D.F.\\
\hbox{\hskip.18cm} {\footnotesize\em e-mail:\/} jap@matem.unam.mx
}
\end{document}